\documentclass[11pt, reqno]{amsart}

\usepackage{amsmath, amssymb, amsthm, amsfonts, amstext, amscd}
\usepackage[all]{xy}
\usepackage{xcolor}
\usepackage{enumerate}
\usepackage{enumitem}
\usepackage{verbatim}
\usepackage[
    linktoc=page,       
    colorlinks=false,   
    pdfborder={0 0 1}   
]{hyperref}
\makeatletter

\newcommand{\Rmnum}[1]{\expandafter\@slowromancap\romannumeral #1@}
\makeatother

\theoremstyle{plain}
\newtheorem{theorem}{Theorem}[section]
\newtheorem{lemma}[theorem]{Lemma}
\newtheorem{proposition}[theorem]{Proposition}
\newtheorem{corollary}[theorem]{Corollary}

\theoremstyle{definition}
\newtheorem{definition}[theorem]{Definition}

\theoremstyle{remark}
\newtheorem{remark}[theorem]{Remark}

\numberwithin{equation}{section}

\title{Symmetric differentials and jets extension of $L^2$ holomorphic functions II: Explicit form}
\date{\today}

\author{Seungjae Lee and Aeryeong Seo}
\address{Department of Mathematics,
Kyungpook National University,
Daegu 41566, Republic of Korea}%
\email{seungjae@knu.ac.kr}
\address{Department of Mathematics,
Kyungpook National University,
Daegu 41566, Republic of Korea}
\email{aeryeong.seo@knu.ac.kr}



\begin{document}

\maketitle
\tableofcontents
\markboth{Seungjae Lee and Aeryeong Seo}{Symmetric differentials and jets extension of $L^2$ holomorphic functions II}

\begin{abstract}
For a symmetric differential on the compact quotient 
$\Sigma = \mathbb{B}^n / \Gamma$ 
of the complex unit ball $\mathbb{B}^n \subset \mathbb{C}^n$ by a discrete subgroup $\Gamma \subset \mathrm{Aut}(\mathbb{B}^n)$, there exists a corresponding weighted $L^2$-holomorphic function on 
$(\mathbb{B}^n \times \mathbb{B}^n)/\Gamma$, 
where $\Gamma$ acts diagonally on $\mathbb{B}^n \times \mathbb{B}^n$. In this paper, we give an explicit description of this correspondence 
and derive several applications based on its explicit form.

\end{abstract}

\section{Introduction}
Let $\mathbb B^n = \{ z \in \mathbb{C}^n : |z| < 1 \}$ be the unit ball in $\mathbb{C}^n$, and let 
$\Gamma \subset \mathrm{Aut}(\mathbb B^n)$ be a cocompact discrete subgroup. 
We denote by $\Sigma = \mathbb B^n / \Gamma$
the corresponding compact complex hyperbolic space form, and $\Omega = (\mathbb B^n \times \mathbb B^n)/\Gamma$
be the quotient under the diagonal action $\gamma \cdot (z,w) = (\gamma z, \gamma w)$. 
Then $\Omega$ is a holomorphic $\mathbb B^n$-fiber bundle over $\Sigma$. 
In previous work~\cite{LS1}, 
the authors studied the relation between holomorphic symmetric differentials on $\Sigma$ and weighted $L^2$ holomorphic functions on $\Omega$.
More precisely, 
authors proved that there exists a natural injective correspondence between symmetric differentials on $\Sigma$ and weighted $L^2$ holomorphic functions on $\Omega$.
For the holomorphic cotangent bundle $T^*_\Sigma$ of $\Sigma$, and its $m$-th symmetric power $S^m T^*_\Sigma$, there exists an injective linear map
$$
\Phi:\ \bigoplus_{m=0}^\infty H^0(\Sigma, S^m T^*_\Sigma)
\;\longrightarrow\;
\bigcap_{\alpha > -1} A^2_\alpha(\Omega) \subset \mathcal{O}(\Omega),
$$
whose image is dense in $\mathcal{O}(\Omega)$ with respect to the compact--open topology.
For $n=1$ this goes back to a result of Adachi~\cite{A21}, 
who obtained an explicit description in terms of the cross ratio on the unit disc.

In the one-dimensional case, if $\Sigma$ is a compact hyperbolic Riemann surface and 
$$\psi = \psi(\tau)(d\tau)^{\otimes N} \in H^0(\Sigma, K_\Sigma^{\otimes N}),$$
Adachi showed that the corresponding function $\Phi(\psi)$ on $(\mathbb B^1 \times \mathbb B^1)/\Gamma$ is given by
$$
\Phi(\psi)(z,w)
=
\begin{cases}
\displaystyle
\frac{1}{B(N,N)}
\int_{\tau \in [z,w]}
\!\!
\left(
\frac{w - z}{(w - \tau)(\tau - z)}
\right)^{N-1}
\psi(\tau)^N\, d\tau, & N \ge 1,\\[1.2em]
\text{the constant }\psi, & N=0,
\end{cases}
$$
where $B(N,N)$ denotes the Beta function.
For $n \ge 2$, however, the construction of $\Phi$ in~\cite{LS1}, which was inspired by Adachi's recursive approach in the one-dimensional case, relied on solving a recursive system of $\bar\partial$-equations 
using the raising operator associated to the K\"ahler form of the Bergman metric, and did not yield a closed formula. 
While the existence and density of $\Phi$ were established, the dependence of $\Phi(\psi)$ on the symmetric differential $\psi$ remained implicit.

The aim of the present paper is to give an explicit formula for $\Phi(\psi)$ in all dimensions $n \ge 1$. 
Let $K(z,w)$ denote the normalized Bergman kernel of $\mathbb B^n$,
$$
K(z,w) = \frac{1}{(1 - z \cdot \overline{w})^{n+1}},
$$
and let $g$ be the normalized Bergman metric on $\mathbb B^n$ with its volume form $dV_g$. 
For each $z,w \in \mathbb B^n$, we define
$$
\varphi_{z,w}(\tau)
= \frac{1}{n+1} \log \frac{K(w,\tau)}{K(z,\tau)}
= \log \frac{1 - z \cdot \tau}{1 - w \cdot \tau},
$$
and let $\nabla \varphi_{z,w}(\tau)$ denote its gradient with respect to $g$.
Our main result is the following.

\begin{theorem}\label{explicit form}
Let $\Sigma = \mathbb B^n/\Gamma$ be a compact complex hyperbolic space form and 
$\Omega = (\mathbb B^n \times\mathbb 
B^n)/\Gamma$ the associated ball bundle. 
For each $N \ge n + 1$ and each $\psi \in H^0(\Sigma, S^N T^*_\Sigma)$, we have
$$
\Phi(\psi)(z,w)
= c_{n,N}
\int_{\mathbb B^n}
\big\langle
(\nabla \varphi_{z,w}(\tau))^{\otimes N}, \psi(\tau)
\big\rangle
\, dV_g(\tau),
$$
for all $(z,w) \in \mathbb B^n \times \mathbb B^n$, where
$$
c_{n,N}
= \frac{(-1)^N (2N - 1)(N - 2)!}{\pi^n (N - n - 1)!}.
$$
\end{theorem}

The integrand in Theorem~1.3 is invariant under the diagonal action of $\mathrm{Aut}(\mathbb  B^n)$, 
and hence $\Phi(\psi)$ descends to a well-defined holomorphic function on $\Omega$. 
From this point of view, the gradient $\nabla \varphi_{z,w}$ plays the role of a vector-valued cross ratio on the unit ball: 
it is built from the Bergman kernel and transforms equivariantly under automorphisms, 
compensating for the absence of a scalar cross-ratio invariant when $n \ge 2$.

The explicit formula for $\Phi$ has several consequences. 
It provides integral representations for Poincar\'e series associated with symmetric differentials (Section~\ref{Poincare series}), 
gives concrete descriptions of weighted Bergman kernels on $\Omega$ using the Bergman kernel of $S^NT^*_\Sigma$ (Section~\ref{Bergman kernels}), 
and extends naturally to the case of totally geodesic isometric holomorphic embeddings into higher-dimensional balls (Section~\ref{isometric embeddings}).

\medskip

The paper is organized as follows. 
In Section~2 we recall the construction of $\Phi$ from~\cite{A21, LS1} together with basic properties of the Bergman kernel and its gradient. 
In Section~3 we compute the infinitesimal behavior of $\Phi(\psi)$ and obtain a formal jet expansion. 
In Section~4 we identify this expansion with the integral expression in Theorem~1.3. 
Section~5 discusses applications to Poincar\'e series and weighted Bergman kernels on $\Omega$, 
and Section~6 extends the formula to totally geodesic isometric embeddings into higher-dimensional balls.

\medskip
{\bf Acknowledgement} 
The first author was supported by the National Research Foundation of Korea (NRF) grant funded by the Korean government(MSIT) (No. RS-2024-00339854). The second author was supported by Basic Science Research Program through the National Research Foundation of Korea (NRF) funded by the Ministry of Education (RS-2025-00561084).

\section{Preliminaries}
\subsection{Construction of $\Phi$}
In this section, we recall some results from \cite{LS1}.

Let $\Gamma$ be a cocompact discrete subgroup of $\mathrm{Aut}(\mathbb B^n)$, and set $\Sigma = \mathbb B^n / \Gamma$, which is a compact complex hyperbolic space form.
Denote by
$$K(z,w) = \frac{1}{(1-z\cdot \overline w)^{n+1}}$$
where $z\cdot \overline w = \sum_{j=1}^n z_j\overline w_j$. Denote $g=(g_{i\overline j})$ denote the normalized Bergman metric, defined by 
$$
g_{i\overline j} = \frac{1}{n+1}\frac{\partial ^2}{\partial z_i\partial \overline z_j} \log K(z,z)
=\frac{(1-|z|^2)\delta_{ij} + z_j\overline z_i}{(1-|z|^2)^2}.
$$
Its inverse $(g^{k\bar\ell})$, satisfying $g^{k\overline\ell} g_{k\overline j} = \delta_{\ell j}$, is given by 
\begin{equation}\label{inverse of g}
g^{k\overline\ell} =  (1-|z|^2)(\delta_{k\ell} - z_k\overline z_\ell).
\end{equation}
Let 
$G$ denote the K\"ahler form associated with 
$g$.

For a fixed point $z\in \mathbb B^n$,
let $T_z$ be an automorphism of $\mathbb B^n$ defined by
\begin{equation}\nonumber
	T_z(w) = \frac{ z- P_z(w) - s_z Q_z(w)}{1- w\cdot \overline z},
\end{equation}
where $s_z = \sqrt{ 1-|z|^2}$ with $|z|^2 = z\cdot \overline z$,
$P_z$ is the orthogonal projection from $\mathbb C^n$ onto the one dimensional subspace $[z]$ generated by $z$,
and $Q_z$ is the orthogonal projection from $\mathbb C^n$ onto $[z]^\perp$.
Let 
\begin{equation}\label{unitary frame}
A=(A_{jk}) := dT_{z} (z) \quad
\text{ and } \quad 
e_j := \sum_k A_{jk}dz_k.
\end{equation}
Then the set $\{ e_1,\ldots, e_n\}$ forms an orthonormal frame of $T_{\mathbb B^n}^{*}$ with respect to $g$ (see \cite[Lemma~2.1]{LS1}), and hence 
$$
g=\sum e_j\otimes \overline e_j
\quad\text{ and } \quad 
G = \sqrt{-1}\sum e_j\wedge \overline e_j.
$$
The raising operator $\mathcal R_G\colon C^\infty(\Sigma, S^m T_\Sigma^*)\to C^\infty(\Sigma, S^{m+1} T_\Sigma^*\otimes \Lambda^{0,1}T_\Sigma^*)) $ is defined by 
$$
\mathcal R_G (\varphi) = \sum_{j=1}^n\varphi e_j\otimes \overline e_j
$$
where $\varphi e_j$ denotes the symmetric product of $\varphi$ and $e_j$.
\medskip

Let $N \geq 1$ be an integer.
For a given symmetric differential $\psi \in H^{0}(\Sigma, S^{N}T_{\Sigma}^{*})$,
define a sequence $\varphi_k\in C^\infty(\Sigma, S^{N+m}T_{\Sigma}^{*})$ with $k=1,2, \ldots $ given by
\begin{equation*}
\left\{ \begin{array}{ll}
\varphi_{k}=0 & \text{if $k<N$}, \\
\varphi_{N}=\psi ,&
\end{array} \right.
\end{equation*}
and for $m \geq 1$,  the section $\varphi_{N+m}$ is the $L^2$ minimal solution of
the following $\overline \partial$-equation:
\begin{equation}\label{d-bar}
\bar \partial \varphi_{N+m} = - (N+m -1) \mathcal R_G\left( \varphi_{N+m-1}  \right).
\end{equation}
Denote $\varphi_m = \sum_{|I|=m}f_I e^I$ with the unitary frame $e=\{e_1,\ldots, e_n\}$ given in \eqref{unitary frame} and define
$$
\Phi(\psi) (z,w):= \sum_{|I|=0}^\infty f_I(z) (T_zw)^I.
$$

\begin{lemma}[\cite{LS1}]
Let $\Box^0_s$ denotes the Laplace operator on $C^\infty(\Sigma, S^{N+s}T^*_\Sigma)$ for each $s$.
Let $\Box^1_s$ and $G_s$ denote the Laplace operator and its Green operator on $C^\infty(\Sigma, S^{N+s}T^*_\Sigma\otimes \Lambda^{0,1}T^*_\Sigma)$, respectively for each $s$.
    \begin{enumerate}
        \item The $L^2$ minimal solution of \eqref{d-bar} is given by
\begin{equation}\label{Step2-2}
\varphi_{N+s}
=-(N+s-1)\bar \partial^{*} G_s\left(\mathcal R_G(\varphi_{N+s-1}  ) \right).
\end{equation}
\item The $L^2$ minimal solution $\varphi_{N+s}$ given in \eqref{Step2-2} is an eigenvector of $\Box^0_s$ for any $s$. Moreover, for its eigenvalue $E_{N,s}$, we have
\begin{equation}\label{G}
 \mathcal R_G \left( \varphi_{N+s}\right)=
\Box^{1} G_s \left( \mathcal R_G\left( \varphi_{N+s}\right)\right)
= G_s \left(
\left( {2 (N+s)} + E_{N,s}
\right)
\mathcal R_G (\varphi_{N+s})\right).
\end{equation}
    \end{enumerate}
\end{lemma}

\medskip

\subsection{Contraction of symmetric differentials with symmetric vector fields}
\begin{definition}(Definition of contraction $\left\langle \,,\right\rangle$)
Let $M$ be a complex manifold and $X$ be a vector field on $M$. Let $\sigma_1, \cdots, \sigma_m$ be sections of $T_M^*$. We define
$$
i_{X} (\sigma_1  \cdots  \sigma_m ) :=  \left\langle X,  \sigma_1  \cdots  \sigma_m \right\rangle = \sum_{k=1}^{m} \left\langle X, \sigma_k \right\rangle \sigma_1  \cdots \widehat \sigma_k   \cdots  \sigma_m
$$
and extend it linearly. We understand
\begin{equation*}
\begin{aligned}
\left\langle X^N, \sigma_1  \cdots  \sigma_m  \right\rangle := \begin{cases} i_{X}^{\circ N} (\sigma_1  \cdots  \sigma_m )  & \text{when} \quad N \leq m,   \\ 0   & \text{when} \quad N > m\end{cases} 
\end{aligned}
\end{equation*}
where $X^N$ is the $N$-th symmetric product of $X$.

\end{definition}

\begin{lemma}\label{comm 1}
Let $\sigma$ is a smooth section of $S^N T_M^*$ and $X$ be a $(1,0)$ vector field on $M$. Then, for any $\gamma\in\textup{Aut}(M)$ and $\tau\in M$,
   $$
    \left\langle (d{\gamma} X)^N, \sigma \right\rangle\circ \gamma(\tau)
    = \left\langle X^N, \gamma^*\sigma\right\rangle(\tau).
    $$ 
\end{lemma}
\begin{proof} 
By definition of the pullback, we have
 \begin{equation}\label{pullback}
    \left\langle d{\gamma} X, \mu \right\rangle\circ \gamma(\tau)
    = \left\langle X, \gamma^*\mu\right\rangle(\tau)
    \end{equation} 
for any section $\mu$ of $T_M^*$. 
Notice that it is enough to prove the lemma for $\sigma=\sigma_1\cdots\sigma_N$ with smooth sections $\sigma_j$, $j=1,\ldots, N$ of $T^*_M$. By \eqref{pullback}, we have
\begin{equation}\nonumber
\begin{aligned}
 \left\langle X^N, \gamma^* (\sigma_1 \cdots \sigma_N) \right\rangle &= i_{X}^N ( \gamma^* (\sigma_1 \cdots \sigma_N) ) \\
 &= i_{X}^N (\gamma^* \sigma_1  \cdots \gamma^* \sigma_N )
 \\
 &= \sum_{k_1=1}^{N} \left\langle X, \gamma^{*} \sigma_{k_1} \right\rangle \gamma^* \sigma_1 \cdots \widehat{\gamma^*\sigma_{k_1}} \cdots  \gamma^* \sigma_N \\
 &= \sum_{k_1=1}^{N} \sum_{k_2 \in [N] - \{k_1 \} }\left\langle X, \gamma^{*} \sigma_{k_1} \right\rangle \left\langle X, \gamma^{*} \sigma_{k_2} \right\rangle \gamma^* \sigma_1  \cdots  \widehat{ \gamma^* \sigma_{k_2} }  \cdots  \widehat{\gamma^*\sigma_{k_1}} \cdots \gamma^* \sigma_N\\
&\qquad\qquad\qquad\qquad\qquad\qquad\vdots \\
&= \sum_{\{k_1,\ldots,k_N\} = [N] }   \left\langle X, \gamma^*\sigma_{k_1} \right\rangle \cdots \left\langle X, \gamma^{*} \sigma_{k_N} \right\rangle\\
&= \sum_{\{k_1,\ldots,k_N\} = [N] }  \left\langle d \gamma X, \sigma_{k_1} \right\rangle (\gamma(\tau)) \cdots \left\langle d \gamma  X, \sigma_{k_N} \right\rangle (\gamma(\tau))  \\
&= {N!}  \left\langle d \gamma X, \sigma_{1} \right\rangle (\gamma(\tau)) \cdots \left\langle d \gamma  X, \sigma_{N} \right\rangle (\gamma(\tau)),
\end{aligned}
\end{equation}
where $[N] = \{1,\ldots, N\}$.
On the other hand, by a similar argument, we obtain
\begin{equation}\nonumber
\left\langle \big(d \gamma X (\sigma (\tau) \big) ^N , (\sigma_1 \cdots  \sigma_N) (\gamma (\tau)) \right\rangle
= N!  \left\langle d \gamma X, \sigma_{1} \right\rangle (\gamma(\tau)) \cdots \left\langle d \gamma  X, \sigma_{N} \right\rangle (\gamma(\tau)).
\end{equation}
Hence we obtain
$$
\left\langle X^N, \gamma^* (\sigma_1 \cdots \sigma_N) \right\rangle (\tau)
= \left\langle \big(d \gamma X (\gamma (\tau) \big) ^N , (\sigma_1 \cdots  \sigma_N) (\gamma (\tau)) \right\rangle,
$$
and it proves the lemma.
\end{proof}

\section{Infinitesimal expression of $\Phi(\psi)$}
\subsection{Formal adjoint operator of $\bar\partial$}
Let $(M,g)$ be an $n$-dimensional hermitian manifold with a volume form $dV_g$ induced by its hermitian metric $g$. 
For a local coordinate $(z_1,\ldots, z_n)$ of $M$, the metric is expressed by $\sum g_{i\bar j} dz_i\otimes d\overline z_j$ and its fundamental form $\omega$ is given by $\omega =  \sqrt{-1} \sum g_{i\bar j} dz_i\wedge d\bar z_j$. The volume form $dV_g = \frac{\omega^n}{n!}$ has a local expression
$$
dV_g  =  \sqrt{-1}^n\det (g_{i\bar j})  dz_1\wedge d \bar z_1\wedge\cdots\wedge dz_n\wedge d\bar z_n.
$$
Let $E$ be a hermitian vector bundle over $M$. Define the Hodge $\star_{E}$-operator  by  
$$
s \wedge \star_{E} t = \left\langle s, t \right\rangle dV_\omega , \quad s, t \in C^{\infty}(\Lambda^{(p,q)} T^*_M \otimes E)
$$
where the wedge product $\wedge$ is combined with the canonical pairing $E \times E^* \rightarrow \mathbb{C}$ and $\left\langle \; , \; \right\rangle$ denotes the inner product induced by the metrics on $E$ and $M$.
Then the formal adjoint operator $\bar \partial^*_{E}$ of $\bar \partial$ is given by
\begin{equation}\label{adjoint}
\bar \partial^*_{E} = - \star_{E^*} \circ \bar \partial \circ \star_{E}.
\end{equation}
See \cite{D} for a general discussion.
For simplicity, we will omit the lower subscript $E$ of $\bar \partial^*_{E}$ if there is no danger of confusion.

\begin{lemma}\label{hodge star for (0,1) form}
Let $(M,g)$ be a hermitian manifold and $E$ be a hermitian vector bundle $M$. Let $\{e_s \}$ be a orthonormal frame of $E$ and $\{e^s\}$ be its coframe on $E^*$. Let $(z_1,\cdots, z_n)$ be a local holomorphic coordinate on $M$.
For any smooth $(0,1)$-form $$
\psi = \sum \psi_{mk} d \bar z_k \otimes e_m,
$$ we have
\begin{equation*}
\star_{E} \psi = \sqrt{-1}^n  \det g \sum_{m,j,k} (-1)^{\text{sgn}+n+k-1} \overline{ \psi_{mj} } g^{j\overline k}   dz\wedge (d \bar z_1 \wedge \cdots \wedge \widehat{d \bar z_k}  \wedge \cdots \wedge d \bar z_n) \otimes e^m 
\end{equation*}
where  $(-1)^{\text{sgn}}$ is given by
$$
dz \wedge d \bar z = (-1)^{\text{sgn}} (dz_1 \wedge d \bar z_1) \wedge \cdots \wedge (dz_n \wedge d \bar z_n).
$$
\end{lemma}
\begin{proof}
Let
$$
\star_{E} \psi = \sum A_{s k}  dz\wedge (d \bar z_1 \wedge \cdots \wedge \widehat{ d \bar z_k} \wedge \cdots  \wedge d \bar z_n) \otimes e^{s}. 
$$
By using the definition of the star operator, we obtain
\begin{equation*}
\begin{aligned}
(d \bar z_j  \otimes e_{m}) \wedge \star_{E} \psi &= \left\langle  d \bar z_j  \otimes e_{m} , \sum_{s,k} \psi_{s k} d \bar z_k \otimes e_{s} \right\rangle \,dV_g \\
&= \sum_{s,k} \delta_{ms} \overline{\psi_{sk}} g^{k\overline j} \, dV_g \\
&=\sum_{k}  \overline{\psi_{m k}} g^{k\overline j} \,dV_g .
\end{aligned}
\end{equation*}
On the other hand, we also have
\begin{equation*}
\begin{aligned}
(d \bar z_j \otimes e_m) \wedge \star_{E} \psi &= \sum_{s,k} A_{sk} \delta_{jk} \delta_{s m} (-1)^{k-1} (-1)^n dz \wedge d \bar z =  A_{mj}  (-1)^{n+j-1} dz \wedge d \bar z.
\end{aligned}
\end{equation*}
Therefore,
$$
 A_{mj}  = (-1)^{\text{sgn}+n+j-1} \sqrt{-1}^n \det g  
 \sum_{k} \overline{\psi_{m k}} g^{k\overline j}  $$
and so the proof is completed.
\end{proof}

\begin{lemma}\label{hodge star for (n,n) form}
Let $(M,g)$ be a hermitian manifold and $E$ be a hermitian vector bundle $M$. Let $\{e_s \}$ be a orthonormal frame of $E$ and $\{e^s\}$ its orthonormal coframe of $E^*$. 
Let $(z_1,\ldots, z_n)$ be a local holomorphic coordinate on $M$.
For any smooth $E$-valued $(n,n)$ form
$$
\psi = \sum_{m} \psi_{m} dz_1 \wedge d \bar z_1 \wedge \cdots \wedge dz_n \wedge d \bar z_n \otimes e_m,
$$
we have
$$
\star_{E} \psi =  \sqrt{-1}^n \sum_m \frac{\overline{\psi_m}}{\det g} e^m.
$$
\end{lemma}

\begin{proof}
By the definition of star operator, we have
\begin{equation}\nonumber
    \begin{aligned}
        &\big( dz_1 \wedge d \bar z_1 \wedge \cdots \wedge dz_n \wedge d \bar z_n \otimes e_s \big) \wedge \star_{E} \psi \\
&= \sum_{m}\delta_{ms}\overline\psi_m \|dz_1 \wedge d \bar z_1 \wedge \cdots \wedge dz_n \wedge d \bar z_n\|^2 dV_g  = \frac{\overline{\psi_s}}{(\det g)^2} dV_g.
    \end{aligned}
\end{equation}
On the other hand, if we put $\star_E\psi = \sum_m A_m e^m$, 
we obtain
\begin{equation}\nonumber
    \begin{aligned}
        \big( dz_1 \wedge d \bar z_1 \wedge \cdots \wedge dz_n \wedge d \bar z_n \otimes e_s \big) \wedge \star_{E} \psi 
        =  A_s dz_1 \wedge d \bar z_1 \wedge \cdots \wedge dz_n \wedge d \bar z_n .
    \end{aligned}
\end{equation}
This implies that $$
A_s = \sqrt{-1}^n\frac{\overline{\psi_s}}{\det g},
$$
and it proves the lemma.
\end{proof}

\begin{lemma}\label{d-bar star}
Let $(M,g)$ be a hermitian manifold and $E$ be a hermitian holomorphic vector bundle. Let $\{e_s \}$ be an orthonormal frame of $E$ and $\{f_s \}$ be a holomorphic frame of $E$. For a holomorphic coordinate $(z_1,\ldots, z_n)$, 
the formal adjoint operator $\bar\partial^*$ for $E$-valued $(0,1)$-form is given by 
    \begin{equation}\nonumber
\begin{aligned}
\bar \partial^{*} \psi &=  - \sum \bigg( \frac{1}{\det g}  \frac{\partial \det g}{ \partial z_k} \psi_{mj}  g^{k\overline j} \overline{\mu^{\zeta m}} \;   +  \frac{\partial \psi_{mj}}{\partial z_k}  g^{k\overline j} \overline{ \mu^{\zeta m} } \;   \\
&\qquad \qquad + \frac{\partial g^{k\overline j}}{\partial z_k}  
 \psi_{mj} \overline{ \mu^{\zeta m} } +   \psi_{mj} g^{k\overline j} {\frac { {\partial \overline{ \mu^{\zeta m} } } } {\partial  z_k} }  
  \bigg) \overline{\mu_{\lambda \zeta}} e_{\lambda}.
\end{aligned}
\end{equation}
 where
$$
\psi = \sum \psi_{mk} d \bar z_k \otimes e_m
$$
and $(\mu^{mk})$ is the inverse matrix of $(\mu_{mk})$ given by  
$$
e_m = \sum_{k} \mu_{mk} f_k.
$$

\end{lemma}
\begin{proof}
 
Let $\{e^s \}$ be the orthonormal coframe of $\{e_s\}$ and $(\tau_{\lambda \zeta})$ the matrix satisfying
$$
e^{\lambda} = \sum \tau_{\lambda \zeta} f^{\zeta}.
$$
where  $\{f^s\}$ is the holomorphic coframe of $\{f_s\}$. Note that since $$
\delta_{\lambda m} = e^{\lambda}(e_m) = \sum \tau_{\lambda \zeta} f^{\zeta} (\sum \mu_{mk}f_k) = \tau_{\lambda \zeta} \mu_{m \zeta}
$$
we obtain
\begin{equation}\label{change of frame 1}
e^{\lambda} = \sum \mu^{\zeta \lambda} f^{\zeta}.
\end{equation}
Now we let
$$
f^{\zeta} = \sum \beta_{\zeta \gamma }e^{\gamma}.
$$
Since
$$
\mu_{\sigma \zeta} = f^{\zeta} (\sum_k \mu_{\sigma k} f_k  ) = f^{\zeta}(e_{\sigma} ) = \sum \beta_{\zeta \gamma} e^{\gamma} (e_{\sigma}) = \beta_{\zeta \sigma}
$$
we have
\begin{equation}\label{change of frame 2}
f^{\zeta} = \sum \mu_{\sigma \zeta} e^{\sigma}.
\end{equation}
Hence, by Lemma~\ref{hodge star for (0,1) form} and \eqref{change of frame 1} we have
\begin{equation*}
\begin{aligned}
\star_{E} \psi &=  \sqrt{-1}^n \det g \sum (-1)^{\text{sgn} + n+k-1} \overline{ \psi_{mj} } g^{j\overline k}  dz \wedge (d \bar z_1 \wedge \cdots \wedge \widehat{d \bar z_k}  \wedge \cdots \wedge d \bar z_n) \otimes e^m   \\
&=\sqrt{-1}^n \det g \sum (-1)^{\text{sgn}+n+k-1} \overline{ \psi_{mj} } g^{j\overline k}  \mu^{\zeta m} dz \wedge (d \bar z_1 \wedge \cdots \wedge \widehat{d \bar z_k}  \wedge \cdots \wedge d \bar z_n) \otimes f^{\zeta}.
\end{aligned}
\end{equation*}
and by \eqref{change of frame 2}
\begin{equation}\nonumber
\begin{aligned}
\bar \partial (\star_{E} \psi) &= \sqrt{-1}^n \sum \frac{ \partial( \det g  \,\overline{\psi_{mj}}\, g^{j\overline k} \mu^{\zeta m}   )  } { \partial \bar  z_k }  (dz_1 \wedge d \bar z_1) \wedge \cdots \wedge (d z_n \wedge d \bar z_n) \otimes f^{\zeta} \\
&= \sqrt{-1}^n \sum  \frac{ \partial( \det g\,  \overline{\psi_{mj}} \,g^{j\overline k} \mu^{\zeta m}   )  } { \partial \bar  z_k } (dz_1 \wedge d \bar z_1) \wedge \cdots \wedge (dz_n \wedge d \bar z_n) \otimes \mu_{\lambda \zeta} e^{\lambda}.
\end{aligned}
\end{equation}
Therefore, by Lemma~\ref{hodge star for (n,n) form}, 
\begin{equation}\nonumber
\begin{aligned}
\star_{E^*} \bar \partial (\star_{E} \psi) &= \star_{E^*} \bigg( \sqrt{-1}^n \sum  \frac{ \partial( \det g  \,\overline{\psi_{mj}} \,g^{j\overline k} \mu^{\zeta m}   )  } { \partial \bar  z_k } (dz_1 \wedge d \bar z_1) \wedge \cdots \wedge (dz_n \wedge d \bar z_n) \otimes \mu_{\lambda \zeta} e^{\lambda} \bigg) \\
&=  \frac{\sqrt{-1}^n}{\det g} \bigg( \overline{\sqrt{-1}^n} \sum \overline{ \frac{ \partial( \det g  \,\overline{\psi_{mj}}\, g^{j\overline k} \mu^{\zeta m}   )  } { \partial \bar  z_k }  } \overline{\mu_{\lambda \zeta}} e_{\lambda} \bigg)  \\
&= \frac{1}{\det g}\sum \bigg(  \frac{\partial \det g}{ \partial z_k} \psi_{mj}  g^{k\overline j} \overline{\mu^{\zeta m}} \;   +  \frac{\partial \psi_{mj}}{\partial z_k} (\det g) g^{k\overline j} \overline{ \mu^{\zeta m} } \;   \\
&\qquad \qquad + \frac{\partial g^{k\overline j}}{\partial z_k}  
(\det g) \psi_{mj} \overline{ \mu^{\zeta m} } \; +  (\det g) \psi_{mj} g^{k\overline j} {\frac { {\partial \overline{ \mu^{\zeta m} } } } {\partial  z_k} }  
  \bigg) \overline{\mu_{\lambda \zeta}} e_{\lambda}. 
\end{aligned}
\end{equation}
Finally, by \eqref{adjoint} we obtain
\begin{equation}\nonumber
\begin{aligned}
\bar \partial^{*} \psi &=  - \sum \bigg( \frac{1}{\det g}  \frac{\partial \det g}{ \partial z_k} \psi_{mj}  g^{k\overline j} \overline{\mu^{\zeta m}} \;   +  \frac{\partial \psi_{mj}}{\partial z_k}  g^{k\overline j} \overline{ \mu^{\zeta m} } \;   \\
&\qquad \qquad \qquad \qquad + \frac{\partial g^{k\overline j}}{\partial z_k}  
 \psi_{mj} \overline{ \mu^{\zeta m} } +   \psi_{mj} g^{k\overline j} {\frac { {\partial \overline{ \mu^{\zeta m} } } } {\partial  z_k} }  
  \bigg) \overline{\mu_{\lambda \zeta}} e_{\lambda}.
\end{aligned}
\end{equation}
\end{proof}

\subsection{Infinitesimal expression of $\Phi(\psi)$}
Now we define, for any $\Xi \in C^{\infty}(\Sigma, S^{\ell} T_{\Sigma}^* ) $,
$$
\mathcal S(\Xi) := (\bar\partial^* \circ \mathcal{R}_{G}) (\Xi).
$$

\begin{corollary}\label{adjoint on ball}
Let $\{e_1,\ldots, e_n \}$ be the unitary frame given in \eqref{unitary frame}. 
Let
$$
\psi = \sum_{|K|=\ell} \psi_K e_K \in C^{\infty} (\Sigma, S^\ell T_{\Sigma}^*).
$$
Then, for any $\alpha\in\mathbb N$
$$
(\mathcal  S^{\alpha} \psi)(0) = (-1)^{\alpha} \sum_{|J|=\alpha}^{} \sum_{|K|=\ell}  \frac{\alpha!}{ J!}{\frac{\partial^\alpha {\psi_{K}}}{\partial z^J}} (0) e^J e^K.
$$

\end{corollary}

\begin{proof}
Since $T_z$ is involutive,  $dT_z (0)$ is the inverse of the matrix $(A_{jk})$ given by \eqref{unitary frame}. 
Moreover, by using the explicit formula \eqref{inverse of g} of $g^{k\bar j}$, it follows that
$$
\frac{\partial^{\alpha} g^{k\overline j}}{\partial z_{\ell_1} \cdots \partial z_{\ell_\alpha}} (0) =0.
$$
Therefore, by repeatably applying Lemma \ref{d-bar star}, and using Lemma 2.1 of \cite{LS1} we obtain the desired result.
\end{proof}

\begin{lemma}
    \begin{equation}\nonumber
    \Phi(\psi) (0,w)= (-1)^N \frac{(2N-1)!}{(N-1)!}
    \sum_{s=0}^\infty {}\frac{(N+s-1)!}{(2N+s-1)!} \sum_{|K|=N} 
    \sum_{|J|=s} 
    \frac{1}{J!}\frac{\partial^s \psi_K}{\partial z^J}(0) w^{J+K}
\end{equation}
\end{lemma}
\begin{proof}
Write $\varphi_{N+s-1}$ as $\sum_{|J|=N+s-1}f_J e^J$ where $\{e_1,\ldots, e_n \}$ is the unitary frame given in \eqref{unitary frame}.  Then
\begin{equation}\nonumber
\begin{aligned}
    \varphi_{N+s}
    &=-\frac{N+s -1}{2(N+s-1) + E_{N,s-1}}
    \bar \partial^{*}  \left( \mathcal R_G(\varphi_{N+s-1}  ) \right)
\end{aligned}
\end{equation}
Now, by repeating the process, for $\psi = \psi_K e^K\in H^0(\Sigma, S^NT^*_\Sigma)$, using Corollary ~\ref{adjoint on ball}, we obtain
\begin{equation}\nonumber
\begin{aligned}
    \varphi_{N+s}(0) 
    &= \prod_{\alpha=1}^{s} \frac{N+\alpha-1}{2(N+\alpha-1)+E_{N,\alpha-1}} \mathcal S^{\alpha} (\varphi_N) \\
    &= (-1)^s \prod_{\alpha=1}^s \frac{N+ \alpha -1}{2(N+\alpha -1) + E_{N,\alpha -1}}
    \sum_{|J|=s} {\frac{s!}{J!} }\frac{\partial^s \psi_K}{\partial z^J}(0) e^Je^K\\
    &= (-1)^s \frac{(N+s-1)!}{(2N+s-1)!}\frac{(2N-1)!}{(N-1)!}
    \sum_{|J|=s} \frac{1}{J!}\frac{\partial^s \psi_K}{\partial z^J}(0) e^Je^K.
    \end{aligned}
\end{equation}
Finally, by $T_0 w = - w$ we have 
\begin{equation}\nonumber
\begin{aligned}
\Phi(\psi) (0,w) &= \sum_{|I|=N}^\infty f_I(0) (-w)^I \\
    &=(-1)^N \frac{(2N-1)!}{(N-1)!}
    \sum_{s=0}^\infty    \frac{(N+s-1)!}{(2N+s-1)!}\sum_{|K|=N}
    \sum_{|J|=s} 
    \frac{1}{J!}\frac{\partial^s \psi_K}{\partial z^J}(0) w^{J+K}.
\end{aligned}
\end{equation}
Therefore, the proof is completed.
\end{proof}

\section{Proof of Theorem~\ref{explicit form}}\label{proof of main theorem}
Let $K\colon\mathbb B^n\times\mathbb B^n\to\mathbb C$ be the Bergman kernel of $\mathbb B^n$. For each $z,\,w\in\mathbb B^n$, define a function $\varphi_{z,w}\colon\mathbb B^n\to \mathbb C$ by 
\begin{equation}\nonumber
    \varphi_{z,w}(\tau) 
    = \frac{1}{n+1}\log \frac{K(w,\tau)}{K(z,\tau)}
    =  \log \left( 
    \frac{1-z\cdot \overline \tau}{1-w\cdot \overline \tau}\right).
\end{equation}
\begin{lemma}\label{trans 2}
    For any $\gamma\in\textup{Aut}(\mathbb B^n)$, it follows that
    $$
    d\gamma^{-1}_\tau\left( \nabla \varphi_{\gamma(z), \gamma(w)}\circ \gamma(\tau) \right)
    = \nabla \varphi_{z,w} (\tau).
    $$
\end{lemma}
\begin{proof}
    By the transformation formula for the automorphism of $\mathbb B^n$, we obtain 
    $$
    \nabla \left(\varphi_{\gamma(z), \gamma(w)}(\gamma(\tau))\right) = \nabla \varphi_{z,w}(\tau).
    $$
    By \cite[Lemma~2.4]{S25},
    $$
    \nabla \left(\varphi_{\gamma(z), \gamma(w)}(\gamma(\tau))\right) 
    = d\gamma_\tau^{-1}\left(\nabla \varphi_{\gamma(z),\gamma(w)}\circ\gamma(\tau)\right),
    $$
    and it proves the lemma.
\end{proof}

For $\psi\in H^0(\Sigma, S^NT^*_\Sigma)$, 
define 
\begin{equation}\nonumber
    \Theta(\psi)(z,w) 
    = \int_{\mathbb B^n}\left\langle \left(
    \nabla \varphi_{z,w}(\tau)\right)^N,
    \psi\right\rangle ~dV_g(\tau).
\end{equation}

\begin{lemma}
For any $\gamma\in \Gamma$, we have 
    $$
    \Theta(\psi)(z,w) = \Theta(\psi)(\gamma z, \gamma w).
    $$
\end{lemma}
\begin{proof} By Lemma~\ref{comm 1} and \ref{trans 2}, we have
\begin{equation}\nonumber
\begin{aligned}
    \Theta(\psi)(z,w) 
    &= \int_{\mathbb B^n}\left\langle \left(
    \nabla \varphi_{z,w}(\tau)\right)^N,
    \psi (\tau) \right\rangle ~dV_g(\tau) \\
    &= \int_{\mathbb{B}^n} \left\langle \big(d \gamma^{-1} \nabla \varphi_{\gamma(z), \gamma(w)} \circ \gamma (\tau))^N, \gamma^* \left(\psi ( \gamma(\tau))\right)  \right\rangle dV_g (\tau) \\
    &= \int_{\mathbb{B}^n} \left\langle (\nabla \varphi_{\gamma(z), \gamma(w)})^N, (\gamma^{-1})^*\gamma^*  \psi \right\rangle dV_g (\tau) \\
    &=\int_{\mathbb{B}^n } \left\langle (\nabla \varphi_{\gamma(z), \gamma(w)})^N,  \psi \right\rangle dV_g (\tau) \\
    &= \Theta(\psi) (\gamma(z), \gamma(w) )
\end{aligned}
\end{equation}
Here, we use $\gamma^{*} \psi = \psi$. 
\end{proof}



\medskip
\begin{proof}[Proof of Theorem~\ref{explicit form}]
For $\gamma \in \mathrm{Aut}(\mathbb{B}^n)$, let $\Gamma' = \gamma^{-1}\Gamma\gamma$ denote the conjugate subgroup of $\Gamma$. Then, for any $\psi \in H^0(\Sigma, S^N T^*_{\Sigma})$, the pullback $\gamma^*\psi$ defines a symmetric differential on $\Sigma' = \mathbb{B}^n / \Gamma'$. Let $\Theta_{\Gamma'}$ and $I_{\Gamma'}$ denote the maps corresponding to $\Theta$ and $I$, respectively, associated with the subgroup $\Gamma'$.
Consider the recursive $\bar\partial$-equation \eqref{d-bar} with base step $\varphi_N = \gamma^*\psi$. Since the Bergman metric is invariant under $\mathrm{Aut}(\mathbb{B}^n)$, we have
$$
\bar\partial(\gamma^*\varphi_{N+m})
= \gamma^*(\bar\partial\varphi_{N+m})
= -(N+m-1)\gamma^*(\mathcal{R}_G \varphi_{N+m-1})
= -(N+m-1)\mathcal{R}_G(\gamma^*\varphi_{N+m-1}),
$$
and hence $\{\gamma^*\varphi_{N+m}\}_{m=0}^{\infty}$ is the solution of the recursive system. This implies that
$$
I_{\Gamma'}(\gamma^*\psi)(z,w)
= I_{\Gamma}(\psi)(\gamma(z), \gamma(w)).
$$
On the other hand,
\begin{equation}\nonumber
\begin{aligned}
\Theta_{\Gamma'}(\gamma^*\psi)(z,w)
&= \int_{\mathbb{B}^n} \left\langle (\varphi_{z,w})^N,\, \gamma^*\psi \right\rangle \, dV_g 
= \int_{\mathbb{B}^n} \left\langle \big( d\gamma(\varphi_{z,w}) \big)^N,\, \psi \right\rangle \, dV_g \\
&= \int_{\mathbb{B}^n} \left\langle \big( \varphi_{\gamma(z), \gamma(w)} \big)^N,\, \psi \right\rangle \, dV_g 
= \Theta_{\Gamma}(\psi)(\gamma(z), \gamma(w)).
\end{aligned}
\end{equation}
As a result, it suffices to prove the theorem in the case $z = 0$.

For the volume form $dV$ of $\mathbb C^n$ with respect to the standard metric, it follows that 
\begin{equation}\nonumber
    \begin{aligned}
        \Theta(\psi)(0,w) 
        &= \int_{\mathbb B^n}\left\langle
        \left(\sum_{\ell,j} (1-|\tau|^2)(\delta_{j\ell} - \tau_j\overline \tau_\ell)\frac{w_\ell}{1-\left\langle w, \tau \right\rangle}\frac{\partial}{\partial \tau_j}\right)^N ,
        \sum \psi_I(\tau)d\tau^I 
        \right\rangle ~\frac{dV}{(1-|\tau|^2)^{n+1}}\\
        &=\int_{\mathbb B^n}
        \frac{(1-|\tau|^2)^{N-n-1}}{(1-\left\langle w, \tau \right\rangle)^N}
        \left\langle
        \left(\sum_j(w_j- \tau_j\left\langle w, \tau\right\rangle)\frac{\partial}{\partial \tau_j}\right)^N,  \sum \psi_I(\tau)d\tau^I
        \right\rangle
        ~dV\\
         &=\int_{\mathbb B^n}
        \frac{(1-|\tau|^2)^{N-n-1}}{(1-\left\langle w, \tau \right\rangle)^N}
        \sum_{|I|=N} \psi_I(\tau)
        \prod_{j=1}^n (w_j-\tau_j\left\langle w, \tau \right\rangle)^{i_j}dV.
        \end{aligned}
        \end{equation}
By expanding each terms as 
$$
\frac{1}{(1-\left\langle w, \tau \right\rangle)^N}
=\sum_{s=0}^\infty \frac{(N+s-1)!}{(N-1)!\, s!}\left\langle w,\tau\right\rangle^s,
$$
$$
\psi_I(\tau) 
= \sum_{k=0}^\infty\sum_{|K|=k}
\frac{1}{K!}\frac{\partial^K\psi_I}{\partial\tau^K }(0)\tau^K,
$$
and for fixed $I$,
$$
\prod_{j=1}^n (w_j-\tau_j\left\langle w, \tau \right\rangle)^{i_j}
=\sum_{r=0}^N \sum_{\substack{A+B=I,\\ |A|=r}}\frac{(-1)^{N-r}I!}{A!\,B!}w^A\tau^B \left\langle w,\tau \right\rangle^{N-r},
$$
by \cite[Lemma~1.11]{Z} we obtain 
        \begin{equation}\nonumber
            \begin{aligned}
            \Theta(\psi)(0,w)
        &=\sum_{|I|=N} \sum_{k=0}^\infty\sum_{|K|=k} \sum_{r=0}^N \sum_{\substack{A+B=I,\\ |A|=r}}\frac{(-1)^{N-r}I!}{A!\,B!}\frac{(N+k-1)!}{(N-1)!\, k!K!}\frac{\partial^K\psi_I}{\partial\tau^K }(0)w^A\\
        &\quad\quad\quad\quad\quad\quad\quad\quad
        \cdot\int_{\mathbb B^n}
        (1-|\tau|^2)^{N-n-1}
        \tau^{K+B} \left\langle w,\tau \right\rangle^{N-r+k}~dV.
        \end{aligned}
        \end{equation}
Here, we note that if \(P \neq Q\), then 
\[
\int_{\mathbb B^n} (1-|\tau|^2)^m\, \tau^{P}\, \overline{\tau}^{Q}\, dV = 0,
\]
and hence only the terms with \(s = k\) survive.
Moreover, by the same reason, only $\overline\tau^{K+B}$ term in $\left\langle w, \tau \right\rangle^{N-r+k}$ survives. As a consequence,  by using the equality 
$$
\int_{\mathbb B^n} (1-|\tau|^2)^m\, \tau^P\,\overline\tau^{\,P}\, dV(\tau)
= \pi^n
\frac{\Gamma(m+1)\,\prod_{j=1}^n \Gamma(p_j+1)}
{\Gamma(m+|P|+n+1)},
$$
we have
        \begin{equation}\nonumber
            \begin{aligned}
            \Theta(\psi)(0,w)
         &=\sum_{|I|=N} \sum_{k=0}^\infty\sum_{|K|=k} \sum_{r=0}^N \sum_{\substack{A+B=I,\\ |A|=r}}\frac{(-1)^{N-r}I!}{A!\,B!}\frac{(N+k-1)!}{(N-1)!\, k!K!}\frac{\partial^K\psi_I}{\partial\tau^K }(0)w^{I+K}\\
        &\quad\quad\quad\quad\quad\quad\quad\quad
        \quad\quad
        \frac{(N-r+k)!}{(B+K)!}\int_{\mathbb B^n}
        (1-|\tau|^2)^{N-n-1}
        \tau^{K+B}  \overline\tau^{B+K}~dV\\
        &=\frac{\pi^n(N-n-1)!}{(N-1)!\, }\sum_{|I|=N} \sum_{k=0}^\infty\sum_{|K|=k} \sum_{r=0}^N \sum_{\substack{A+B=I,\\ |A|=r}}\frac{(-1)^{N-r}I!}{A!\,B!}
        \frac{(N-r+k)!}{(2N-r+k-1)!}\\
        &\quad\quad\quad\quad\quad\quad
        \quad\quad\quad\quad\quad\quad\quad\quad
        \cdot \frac{(N+k-1)!}{k!K!}\frac{\partial^K\psi_I}{\partial\tau^K }(0)w^{I+K}\\
        \end{aligned}
        \end{equation}
By exploiting 
$$
 \sum_{\substack{A+B=I,\\ |A|=r}}\frac{I!}{A!\,B!} = \binom{N}{r}
$$
and 
\begin{equation}\nonumber
\begin{aligned}
&\sum_{r=0}^N (-1)^{N-r}\binom{N}{r}\frac{(N-r+k)! }{(2N-r+k-1)!}\\
& = \sum_{r=0}^{N} (-1)^{N-r} \binom{N}{N-r} \frac{\big((N-r)+k \big)!} 
{(N+(N-r)+k-1)!}  \\
&= \sum_{s=0}^{N} (-1)^s \binom{N}{s} \frac{(s+k)!} 
{(N+s+k-1)!} \\
&= \sum_{s=0}^{N} (-1)^s \binom{N}{s} \frac{1}{(N-2)!} \int_0^{1} t^{s+k} (1-t)^{N-2} dt \\
&= \frac{1}{(N-2)!} \int_0^{1}t^k \bigg(  \sum_{s=0}^{N}  \binom{N}{s} (-t)^s  \bigg) (1-t)^{N-2} dt \\
&= \frac{1}{(N-2)!} \int_0^{1} t^k (1-t)^{2N-2} dt \\
&= \frac{k!(2N-2)!}{(N-2)!(2N+k-1)!},
\end{aligned}
\end{equation}
we obtain 
        \begin{equation}\nonumber
            \begin{aligned}
            \Theta(\psi)(0,w)
            &=\frac{\pi^n(N-n-1)!(2N-2)!}{(N-1)!\,(N-2)! }\sum_{|I|=N} \sum_{k=0}^\infty\sum_{|K|=k} 
        \frac{(N+k-1)!}{(2N+k-1)!K!}\frac{\partial^K\psi_I}{\partial\tau^K }(0)w^{I+K}.
    \end{aligned}
\end{equation}
\end{proof}

\section{Application of the explicit form of $\Phi$}
\subsection{Poincar\'e series}\label{Poincare series}
In \cite[for $n=1$]{O16} and \cite[pp. 1265--1266 for $n\geq 2$]{LS1}, it is proved that the Poincar\'e series defined by
\begin{equation}\nonumber
	\sum_{\gamma\in \Gamma}  (\gamma_j(z)-\gamma_j(w))^{N}
\end{equation}
is a $\Gamma$-invariant holomorphic function on $\mathbb B^n\times\mathbb B^n$ for any $N\geq n+1$
with respect to the diagonal action, i.e.
this series converges when $N\geq n+1$.
Now, we consider a holomorphic coordinate system $(\widetilde z, \widetilde w)$ on $\mathbb{B}^n \times \mathbb{B}^n$ which is given by $\widetilde z := z$ and $\widetilde w := w-z$.  By the Taylor expansion of $\gamma_j (\tilde z + \tilde w)$ at $\tilde w=0$, we have
\begin{equation*}
\begin{aligned}
&\sum_{\gamma \in \Gamma} (\gamma_j (z) - \gamma_j (w))^N 
= \sum_{\gamma \in \Gamma} (\gamma_j (\tilde z ) - \gamma_j (\tilde z + \tilde w) )^N \\
&= (-1)^N \sum_{|I|=N}\frac{N!}{I!} \bigg(  \sum_{\gamma \in \Gamma} \bigg( \frac{\partial \gamma_j (z) } {\partial z_1} \bigg)^{i_1} \cdots \bigg( \frac{\partial \gamma_j (z) } {\partial z_n} \bigg)^{i_n} \bigg)  {\tilde w}^{i_1} \cdots {\tilde w}^{i_n} + O( {| \tilde{w} |}^{N+1} ).
\end{aligned}
\end{equation*}
Therefore,
\begin{equation}\nonumber
\sum_{\gamma\in\Gamma} \left(
\frac{\partial\gamma_j}{\partial z_1}\right)^{i_1}\cdots \left(
\frac{\partial\gamma_j}{\partial z_n}\right)^{i_n}
\end{equation}
with any $I=(i_1,\ldots, i_n)$ such that $i_1 + \cdots + i_n = N$ converges.
This implies that 
\begin{equation}\nonumber
    \sum_{\gamma \in \Gamma}  \gamma^*   dz_j^N
    =\sum_{\gamma\in\Gamma} \left(
    \sum_{\ell=1}^n \frac{\partial \gamma_j}{\partial z_\ell}dz_\ell\right)^N
    = \sum_{\gamma\in\Gamma} \sum_{|I|=N}\frac{N!}{I!}
    \left(
\frac{\partial\gamma_j}{\partial z_1}\right)^{i_1}\cdots \left(
\frac{\partial\gamma_j}{\partial z_n}\right)^{i_n}dz^I
\end{equation}
converges.
\begin{proposition}
    For any $N\geq n+1$, the Poincar\'e series
    $$\sum_{\gamma \in \Gamma}  \left(\gamma^*   dz_j\right)^N$$ converges.
\end{proposition}
\begin{theorem}
Let $\Gamma$ be a torsion-free discrete cocompact subgroup of the automorphism group of the unit ball $\mathbb{B}^n$. For any $N \geq n+1$ and $j\in\{1,\ldots, n\}$,
$$
\Phi \big(\sum_{\gamma \in \Gamma}  \gamma^*   d\tau_j^N   \big)(z, w) 
= D_{n,N}
\sum_{\gamma \in \Gamma} (\gamma_j (z) - \gamma_j (w))^N
$$
with $D_{n,N}=\frac{(2N-1)(N-n)}{2N (n-1)(N-1)}$ and $\gamma = (\gamma_1,\ldots, \gamma_n)$.
\end{theorem}
\begin{proof}
For fixed $z,\, w\in\mathbb B^n$, since we have
\begin{equation}\nonumber
    \begin{aligned}
        \nabla \varphi_{z,w}(\tau) 
        &= \frac{1}{n+1}\sum_{i,j} g^{i\bar j }(\tau)\frac{\partial}{\partial \bar\tau_j} \log \frac{K(w,\tau)}{K(z,\tau)}\frac{\partial}{\partial\tau_i}\\
        &=(1-|\tau|^2) \sum_{i,j} (\delta_{ij}-\tau_i\bar\tau_j)
        \left(\frac{-z_j}{1-\langle z,\tau\rangle} + \frac{w_j}{1-\langle w,\tau \rangle}\right)
        \frac{\partial}{\partial\tau_i}\\
         &=(1-|\tau|^2) \sum_{i} 
        \left(\frac{-z_i}{1-\langle z,\tau\rangle} + \frac{w_i}{1-\langle w,\tau \rangle}
        -\tau_i\left(
        \frac{-\langle z,\tau\rangle}{1-\langle z,\tau\rangle} + \frac{\langle w,\tau\rangle}{1-\langle w,\tau \rangle}\right)\right)
        \frac{\partial}{\partial\tau_i}\\
         &=(1-|\tau|^2) \sum_{i} 
        \left(\frac{-z_i}{1-\langle z,\tau\rangle} + \frac{w_i}{1-\langle w,\tau \rangle}
        -\tau_i\left(
        \frac{-1}{1-\langle z,\tau\rangle} + \frac{1}{1-\langle w,\tau \rangle}\right)\right)
        \frac{\partial}{\partial\tau_i}\\
        &=(1-|\tau|^2) \sum_{i} 
        \left(\frac{-z_i+\tau_i}{1-\langle z,\tau\rangle} + \frac{w_i-\tau_i}{1-\langle w,\tau \rangle}\right)
        \frac{\partial}{\partial\tau_i},
    \end{aligned}
\end{equation}
we obtain
\begin{equation}\nonumber
    \begin{aligned}
        \left\langle (\nabla \varphi_{z,w}(\tau) )^N, d\tau_j^N\right\rangle
        &=(1-|\tau|^2)^N \left\langle\left(\sum_{i} 
        \left(\frac{-z_i+\tau_i}{1-\langle z,\tau\rangle} + \frac{w_i-\tau_i}{1-\langle w,\tau \rangle}\right)
        \frac{\partial}{\partial\tau_i}\right)^N 
        , d\tau_j^N \right\rangle\\
        &=(1-|\tau|^2)^N 
        \left(\frac{-z_j+\tau_j}{1-\langle z,\tau\rangle} + \frac{w_j-\tau_j}{1-\langle w,\tau \rangle}\right)^N
    \end{aligned}
\end{equation}
and therefore 
\begin{equation}\nonumber
    \begin{aligned}
       &\int_{\mathbb B^n} \left\langle (\nabla \varphi_{z,w}(\tau) )^N, d\tau_j^N\right\rangle dV_g\\
        &=\int_{\mathbb B^n} (1-|\tau|^2)^{N-n-1} \left\langle\left(\sum_{i} 
        \left(\frac{-z_i+\tau_i}{1-\langle z,\tau\rangle} + \frac{w_i-\tau_i}{1-\langle w,\tau \rangle}\right)
        \frac{\partial}{\partial\tau_i}\right)^N 
        , d\tau_j^N \right\rangle dV\\
        &=\int_{\mathbb B^n}(1-|\tau|^2)^{N-n-1} 
        \left(\frac{-z_j+\tau_j}{1-\langle z,\tau\rangle} + \frac{w_j-\tau_j}{1-\langle w,\tau \rangle}\right)^N dV \\
        &= (w_j - z_j)^N\int_{\mathbb B^n}(1-|\tau|^2)^{N-n-1} 
        (1-|\tau_j|^2)^N dV \\
        &=\frac{ \pi^n(N-n)!}{2 (n-1)N!}(w_j - z_j)^N.
    \end{aligned}
\end{equation}
For the third equality, we expand the denominators in power series and use the fact that $$\int_{\mathbb B^n} (1-|\tau|^2)^m \tau^P\overline\tau^QdV=0$$
for any $P\neq Q$.
As a consequence, 
\begin{equation}\nonumber
    \begin{aligned}
        \Phi\left(\sum_{\gamma\in\Gamma}\gamma^* \left(d\tau_j^N\right)\right)(z,w) 
        &=c_{n,N}\int_{\mathbb B^n}\left\langle\left(
    \nabla \varphi_{z,w}(\tau)\right)^N, 
   \sum_{\gamma\in\Gamma}\gamma^* \left(d\tau_j^N\right)\right\rangle ~dV_g(\tau)\\
   &=c_{n,N} \sum_{\gamma\in\Gamma}\int_{\mathbb B^n}\left\langle\left(
    \nabla \varphi_{z,w}(\tau)\right)^N, 
  \gamma^* \left(d\tau_j^N\right)\right\rangle ~dV_g(\tau)\\
  &=c_{n,N} \sum_{\gamma\in\Gamma}\int_{\mathbb B^n}\left\langle\left(
    d\gamma\nabla \left(\varphi_{z,w}(\tau)\right)\right)^N, 
   d\tau_j^N\right\rangle \circ\gamma(\tau)~dV_g(\tau)\\
   &=c_{n,N} \sum_{\gamma\in\Gamma}\int_{\mathbb B^n}\left\langle\left(
    \nabla \varphi_{\gamma(z),\gamma(w)}(\tau)\right)^N, 
   d\tau_j^N\right\rangle \circ\gamma(\tau)~dV_g(\tau)\\
   &=c_{n,N} \sum_{\gamma\in\Gamma}\int_{\mathbb B^n}\left\langle\left(
    \nabla \varphi_{\gamma(z),\gamma(w)}(\tau)\right)^N, 
   d\tau_j^N\right\rangle (\tau)~dV_g(\tau)\\
        &= c_{n,N}\frac{ \pi^n(N-n)!}{2 (n-1)N!}\sum_{\gamma\in\Gamma}(\gamma_j(w) - \gamma_j(z))^N.
    \end{aligned}
\end{equation}
The third equality follows from Lemma~\ref{comm 1}, the fourth from Lemma~\ref{trans 2}, and the fifth from the invariance of the Bergman volume form.
\end{proof}
\begin{remark}
    For $n=1$, it is proved in \cite[Corollary 4]{A21}.
\end{remark}

\subsection{Relation between the Bergman kernel on $H^0(\Sigma,S^NT^*_\Sigma)$ and $A^2_\alpha(\Omega)$}\label{Bergman kernels}

In this section, we define the weighted Bergman space $A^2_{\alpha}(\Omega)$ on $\Omega$ and express the weighted Bergman kernel of $A^2_{\alpha}(\Omega)$ using Bergman kernels of $H^0(\Sigma, S^m T_{\Sigma}^*)$. 

To define the Bergman kernel $B_m$ of $H^0(\Sigma, S^m T_{\Sigma}^*)$, we consider an orthonormal basis 
$$
\bigg \{ \psi_{m,j} = \sum_{ | I | =m }  \psi_{m,I,j} (\tau) d\tau^I  \bigg \}_{j=1}^{d_m}
$$
of $H^0(\Sigma, S^m T_{\Sigma}^*)$ as a Hilbert space equipped with the $L^2$-inner product induced by the Bergman metric on $\Sigma$, where $d_m:=\dim H^0 (\Sigma, S^m T_{\Sigma}^*) <\infty$. 
Then the Bergman kernel of $H^0(\Sigma, S^m T_{\Sigma}^*)$ is defined by
\begin{equation}\nonumber
\begin{aligned}
B_m (\tau, \tau') &:= \sum_{j=1}^{d_m} \psi_{m,j}(\tau) \otimes  \overline{ \psi_{m,j} (\tau') }. 
\end{aligned}
\end{equation}
Note that it is independent of a choice of orthonormal basis on $H^0(\Sigma, S^m T_{\Sigma}^*)$. 

Let $\omega$ be a K\"ahler form on $\Omega$ defined by
\begin{equation}\nonumber
\omega= \frac{\sqrt{-1}}{n+1} \partial \bar \partial \log K(z,z) + \frac{\sqrt{-1}}{n+1} \partial \bar \partial \log K(w,w),
\end{equation}
and let {$dV_{\omega} = \frac{1 }{(2n)!}\omega^{2n} $ be its volume form. For measurable sections $f_1$, $f_2$ on $\Lambda^{p,q} T_{\Omega_{}}^{*}$ and $\alpha > -1$,
we set
\begin{equation}\nonumber
\langle \langle f_1,f_2 \rangle \rangle_{\alpha} : = c_{\alpha} \int_{\Omega} \langle f_1, f_2 \rangle_{\omega} \delta^{\alpha+n+1} \,dV_\omega
\end{equation}
where $c_{\alpha} = \frac{\Gamma(n+\alpha+1)}{ \Gamma(\alpha+1) n! }$ and $\delta = 1-|T_{z} w|^2$.

For $\alpha > -1$, we define a weighted $L^2$-space by setting
\begin{equation}\nonumber
L^2_{(p,q), \alpha}(\Omega) := \{ f : f \text{ is a measurable section on $\Lambda^{p,q} T_{\Omega_{\rho}}^*$}, ~\|f \|^2_{\alpha} := \left< f,f\right>_{\alpha} < \infty \}
\end{equation}
and a weighted Bergman space by
$A^2_{\alpha}(\Omega) := L^2_{(0,0), \alpha}(\Omega) \cap \mathcal O(\Omega).$

\begin{theorem}\label{integral formula of Bergman Kernel}
Let $B_m$ denote the Bergman kernel of $H^0(\Sigma, S^m T^*_\Sigma)$, 
and let $\{\psi_{m,j}\}_{j=1}^{d_m}$ be an orthonormal basis of this space.
For $\alpha>-1$, the weighted Bergman kernel $K_{\alpha}$ of $A^2_{\alpha}(\Omega)$ is given by
\begin{equation}\nonumber
\begin{aligned}
&K_{\alpha}\big((z,w),(z',w')\big) \\
&= S_n +  \sum_{m=n+1}^{\infty} \frac{(c_{n,m})^2}{d_{\alpha,m}} \int_{\mathbb{B}^{n}}  \bigg\langle {( \nabla \varphi_{z,w} (\tau) )}^m ,  \overline{\bigg( \int_{\mathbb{B}^n} \left\langle (\nabla \varphi_{z',w'} (\tau') )^m, \overline{B_m (\tau, \tau')}  \right \rangle dV_g (\tau') \bigg) }  \bigg \rangle  dV_g(\tau) 
\end{aligned}
\end{equation}
where 
$$
S_n := \frac{n!}{\pi^n 2^{2n} (\mathrm{Vol} (\Sigma) )} +  \sum_{m=1}^{n} \sum_{j=1}^{d_m} \frac{ \Phi(\psi_{m,j})(z,w)     \overline{\Phi(\psi_{m,j})(z',w')} }{ d_{\alpha,m} }
$$
and 
$$
d_{\alpha,m}:=  \frac{ \pi^n 2^{2n}}{n!} \frac{ \Gamma(n+\alpha+1) \Gamma(m+1) }{ \Gamma (m+n + \alpha +1) }  \sum_{\ell=0}^{\infty} \frac{(m+1)_{\ell}}{(n+m+\alpha+1)_{\ell}} \frac{(m)_{\ell} (m)_{\ell}}{(2m)_{\ell}} \frac{1}{\ell !}  \prod_{j=1}^{\ell} \bigg(1+ \frac{n-1}{m+j} \bigg) 
$$
where $(m)_{\ell}:= m (m+1)\cdots(m+\ell-1)$.

\end{theorem}

\begin{proof}
We express the orthonormal basis $\{\psi_{m,j} \}_{j=1}^{d_m}$ by
$$
 \psi_{m,j} = \sum_{ | I | =m }  \psi_{m,I,j} (\tau) d\tau^I.
$$
Then, the Bergman kernel $B_m$ is expressed by
\begin{equation}\nonumber
\begin{aligned}
B_m (\tau, \tau') &= \sum_{j=1}^{d_m} \bigg(  \sum_{|I|=m} \psi_{m,I,j} (\tau) d \tau^I  \bigg) \otimes \overline { \bigg( \sum_{|K|=m} \psi_{m,K,j} (\tau') d \tau'^K \bigg)} \\
&= \sum_{j=1}^{d_m} \sum_{|I|=m, |K|=m} \psi_{m,I,j}(\tau) \overline{\psi_{m,K,j} (\tau') } d \tau^I \otimes d \overline{\tau'^K}.
\end{aligned}
\end{equation}
Let
\begin{equation}\label{definition of S}
\begin{aligned}
S_{m,j}^K (z',w') &:= \int_{\mathbb{B}^n} \left\langle (\nabla \varphi_{z',w'} (\tau') )^m, \psi_{m,K,j}(\tau') d\tau'^K \right \rangle dV_g (\tau')  \\
&= \int_{\mathbb{B}^n} (1-|\tau'|^2)^{m-(n+1)} \prod_{\ell=1}^{n}  \bigg( \frac{-z'_\ell + \tau'_\ell}{1-\langle z', \tau' \rangle} + \frac{w'_\ell- \tau'_\ell} {1-\langle w', \tau' \rangle} \bigg)^{k_\ell} \psi_{m,K,j}(\tau')  dV_g (\tau'). 
\end{aligned}
\end{equation}
Then,
\begin{equation*}
\begin{aligned}
& \int_{\mathbb{B}^n} \left\langle (\nabla \varphi_{z',w'} (\tau') )^m, \overline{B_m (\tau, \tau')} \right\rangle dV_g (\tau') \\
&= {   \int_{\mathbb{B}^n} \left \langle (\nabla \varphi_{z',w'} (\tau') )^m, \sum_{j=1}^{d_m} \bigg( \sum_{|K|=m} \psi_{m,K,j} (\tau') d\tau'^K \bigg) \otimes\overline{ \bigg( \sum_{|I|=m}  \psi_{m,I,j} (\tau) d \tau^I  \bigg) } \right \rangle dV_g (\tau')  }     \\
&= \sum_{j=1}^{d_m} \overline{ \bigg( \sum_{|I|=m} \psi_{m,I,j} (\tau) d \tau^I  \bigg) } \int_{\mathbb{B}^n}  \bigg \langle ( \nabla \varphi_{z',w'} (\tau') )^m,  { \sum_{| K |=m}  \psi_{m,K,j} (\tau') d \tau'^K   } \bigg \rangle dV_g (\tau') \\
&= \sum_{j=1}^{d_m} \overline{ \bigg( \sum_{|I|=m}  \psi_{m,I,j} (\tau) d \tau^I  \bigg) } \sum_{| K |=m}   \int_{\mathbb{B}^n}  \bigg \langle ( \nabla \varphi_{z',w'} (\tau') )^m,  {  \psi_{m,K,j} (\tau') d \tau'^K   } \bigg \rangle dV_g (\tau') \\
&= \sum_{j=1}^{d_m} \overline{ \bigg( \sum_{|I|=m}  \psi_{m,I,j} (\tau) d \tau^I  \bigg) }   \cdot \bigg( \sum_{|K|=m} S_{m,j}^K (z',w') \bigg).
\end{aligned}
\end{equation*}
Therefore, by \eqref{definition of S}
\begin{equation}\label{intergral of bergman kernel}
\begin{aligned}
&\int_{\mathbb{B}^{n}}  \bigg\langle {( \nabla \varphi_{z,w} (\tau) )}^m ,  \overline{\bigg( \int_{\mathbb{B}^n} \left\langle (\nabla \varphi_{z',w'} (\tau') )^m, \overline{B_m (\tau, \tau')}  \right\rangle dV_g (\tau') \bigg)} \bigg \rangle  dV_g(\tau)  \\
&= \sum_{j=1}^{d_m} \int_{\mathbb{B}^{n}}  \bigg\langle {( \nabla \varphi_{z,w} (\tau) )}^m , \overline{\bigg(  \overline{ \bigg( \sum_{|I|=m}  \psi_{m,I,j} (\tau) d \tau^I  \bigg) }   \cdot \bigg( \sum_{|K|=m} S_{m,j}^{K} (z',w')  \bigg)}     \bigg \rangle  dV_g(\tau)  \\
&=\sum_{j=1}^{d_m} \overline{ \bigg( \sum_{|K|=m } S_{m,j}^K (z',w')  \bigg) } \cdot \int_{\mathbb{B}^n}  \bigg\langle {( \nabla \varphi_{z,w} (\tau) )}^m ,  { \sum_{|I|=m}  \psi_{m,I,j} (\tau) d \tau^I   }   \bigg \rangle dV_g (\tau) \\
&= \sum_{j=1}^{d_m} \bigg( \sum_{|I|=m} S_{m,j}^I(z,w) \bigg) \cdot \overline{ \bigg (\sum_{|K|=m} S_{m,j}^K(z',w') \bigg)} \\
&= \sum_{j=1}^{d_m}   \int_{\mathbb{B}^n} \left\langle (\nabla \varphi_{z,w} (\tau) )^m, \psi_{m,j} (\tau) \right\rangle dV_g (\tau)  \cdot  \overline{  \int_{\mathbb{B}^n} \left\langle (\nabla \varphi_{z',w'} (\tau') )^m, \psi_{m,j} (\tau') \right\rangle dV_g (\tau')}  .
\end{aligned}
\end{equation}

On the other hand, by the proof of \cite[Lemma 4.14, Corollary 4.15, and Lemma 4.17]{LS1}, 
$$
\bigcup_{m=0}^{\infty} \bigg \{ \frac { \Phi(\psi_{m,j}) } { \| \Phi(\psi_{m,j} ) \|_{\alpha} }  \bigg\}_{j=1}^{d_m}
$$
is a complete orthonormal basis on $A^2_{\alpha}(\Omega)$ when $\alpha>-1$. Therefore, by Theorem \ref{explicit form} and \eqref{intergral of bergman kernel}, $K_{\alpha}((z,w),(z',w'))$ is equal to
\begin{equation}\nonumber
\begin{aligned}
S_n + \sum_{m=n+1}^{\infty} \frac{(c_{n,m})^2}{d_{\alpha,m}} \int_{\mathbb{B}^{n}}  \bigg\langle {( \nabla \varphi_{z,w} (\tau) )}^m , \overline{ \bigg( \int_{\mathbb{B}^n} \bigg\langle (\nabla \varphi_{z',w'} (\tau') )^m, \overline{B_m (\tau, \tau')}  \bigg\rangle dV_g (\tau') \bigg) }  \bigg \rangle  dV_g(\tau) 
\end{aligned}
\end{equation}
and the proof is completed.
\end{proof}


\section{Explicit formula for totally geodesic isometric embeddings}\label{isometric embeddings}
In this section, we derive an explicit formula for the map~$\Phi$ associated with the holomorphic $\mathbb{B}^N$-fiber bundle over $M$, where the compact complex manifold $M$ admits a totally geodesic, isometric, equivariant holomorphic embedding into $\mathbb{B}^N$. The structure of the argument follows the approach used in Section~\ref{proof of main theorem}.

For a complex manifold $\widetilde M$, a lattice $\Gamma$ in $\textup{Aut}(\widetilde M)$ and a homomorphism $\rho\colon\Gamma\to\textup{Aut}(\mathbb B^N)$, we say that a map $f\colon \widetilde M\to\mathbb B^N$ is $\rho$-equivariant if  $f(\gamma z) = \rho(\gamma)f(z)$ holds for any $\gamma\in\Gamma$ and $z\in \widetilde M$.

\begin{theorem}[\cite{LS2}]\label{totally geodesic imbedding}
Let $\widetilde{M}$ be a complex manifold, $\Gamma$ be a torsion-free cocompact lattice of $\textup{Aut}(\widetilde{M})$ and $\rho\colon\Gamma\to SU(N,1)$ be a representation.
Suppose that there exists a $\rho$-equivariant totally geodesic isometric holomorphic embedding $\imath\colon \widetilde M\to\mathbb B^N$. Let
$M:=\widetilde M/\Gamma$ and $\Sigma:=\mathbb B^N/\rho(\Gamma)$.
Let $\Omega_{\rho}:=M\times_\rho \mathbb B^N$ be a holomorphic $\mathbb B^N$-fiber bundle over $M$ where any $\gamma\in \Gamma$ acts on $\widetilde M\times \mathbb B^N$ by $(\zeta,w)\mapsto (\gamma \zeta, \rho(\gamma) w)$. Then there exists an injective linear map
\begin{equation}\nonumber
\Phi: \bigoplus_{m=0}^{\infty} H^0 (M, \imath^* (S^m T_{\Sigma}^*)) \rightarrow
\begin{cases}
\displaystyle\bigcap_{\alpha>-1} A^2_\alpha (\Omega_{\rho}) \subset \mathcal{O} (\Omega_{\rho}) & \text{ if } n=N,\\
\displaystyle\bigcap_{\alpha\geq -1} A^2_\alpha (\Omega_{\rho}) \subset \mathcal{O} (\Omega_{\rho})  & \text{ if } n< N,\\
\end{cases}
\end{equation}
which has a dense image in $\mathcal{O}(\Omega_{\rho})$ equipped with the compact open topology. In particular, 
\[
\begin{cases}
\dim A^2_{\alpha}(\Omega_{\rho}) = \infty, & \text{if } \alpha > -1, \\[6pt]
A^2_{-1}(\Omega_{\rho}) = \displaystyle\bigcap_{\alpha \ge -1} A^2_{\alpha}(\Omega_{\rho}), \quad 
\dim A^2_{-1}(\Omega_{\rho}) = \infty, & \text{if } N > n.
\end{cases}
\]

\end{theorem}

For the definition of $A^{2}{\alpha}(\Omega_{\rho})$, see \cite[pp.~20]{LS2}. Although, in \cite{LS2} the symbol $\Omega$ is used for $\Omega_{\rho}$, we will use the notation $\Omega_{\rho}$ in this section to avoid confusion with the notation used in the previous sections.

Let $N = \imath^* T_{\Sigma} / T_M$ be the holomorphic normal bundle of $M$. Since $M$ is totally geodesically embedded in $\Sigma$, the normal bundle $N$ is holomorphically isomorphic to the orthogonal complement of $T_M$ in $\imath^* T_{\Sigma}$ with respect to the induced metric from the normalized Bergman metric of $\mathbb{B}^N$. As a result, the pullback of the symmetric power of the cotangent bundle of $M$ admits the following holomorphic decomposition:
$$
H^0(M, \imath^* (S^m T_{\Sigma}^*)) \cong \bigoplus_{\ell=0}^{m} H^0 (M, S^\ell T_{M}^* \otimes S^{m-\ell} N^*).
$$
$$
$$

\begin{theorem}
For $\psi^{\ell}_m \in H^0(M, S^\ell T_{M}^* \otimes S^{m-\ell} N^*)$, we have 
    \begin{equation}\nonumber
    \begin{aligned}
        \Phi ( \psi^\ell_m )(\zeta, w)
        &= \frac{c_{n,m}{(m+\ell-1)}}{2m-1}\int_{\imath(\widetilde M)}
        \left\langle
        \left.\left(
        \sum_{j,\ell} g^{j\bar \ell}(\tau)\frac{\partial}{\partial\overline\tau_\ell}\frac{K(\imath(\zeta),\tau)}{K(w,\tau)}\frac{\partial}{\partial\tau_j}\right)^m\right|_{\tau\in\imath(\widetilde M)},
        \psi^\ell_m \right\rangle ~dV_{\widetilde M}
   \end{aligned}
    \end{equation}
    where $dV_{\widetilde M}$ is the volume form on $\imath(\widetilde M)$ with respect to the induced metric from $\mathbb B^N$, $g^{j\bar\ell}$ is the inverse of the Bergman metric of $\mathbb B^N$ and $K$ is the Bergman kernel of $\mathbb B^N$.
    
\end{theorem}
 
\begin{proof}
Define $\Theta$ to be an operator from $H^0(M,S^\ell T_M^*\otimes S^{m-\ell}N^*)$ to $A_\alpha^2(\Omega_{\rho})$ by 
$$\Theta(\psi_m^\ell)(\zeta,w):=
\int_{\imath(\widetilde M)}
        \left\langle
        \left.\left(
        \sum_{j,\ell} g^{j\bar \ell}(\tau)\frac{\partial}{\partial\overline\tau_\ell}\frac{K(\imath(\zeta),\tau)}{K(w,\tau)}\frac{\partial}{\partial\tau_j}\right)^m\right|_{\tau\in\imath(\widetilde M)},
        \psi^\ell_m \right\rangle ~dV_{\widetilde M}.$$
By a similar way, one can show that $\Theta(\psi^\ell_m)$ is invariant under the diagonal action of the conjugate subgroup of $\Gamma$. Since
 $\widetilde M$ is totally geodesically embedded in $\mathbb{B}^n$, without loss of generality, we may assume that 
\begin{equation}\label{normalize 1}
\imath(\widetilde M) = \{(\tau_1,\cdots, \tau_N) \in \mathbb{B}^N : \tau_{n+1}=\cdots=\tau_N=0 \}
\end{equation}
and $$
\quad \imath (\zeta)=0,
$$
where $(\tau_1,\cdots, \tau_N)$ is the standard Euclidean coordinate of $\mathbb{B}^N$.

In what follows, we use the multiindex $J$ to indicate {normal-direction} indices.
For example, under the setting \eqref{normalize 1}, if $w=(w_1,\ldots,w_N)$ and 
$J=(j_{n+1},\ldots,j_{N})$, then
$$
    w^{J} := w_{n+1}^{\, j_{n+1}} \cdots w_{N}^{\, j_{N}}.
$$
Other multiindices, which are denoted by $L$, $A$, $B$, and $K$, will refer to {tangent-direction} indices.

Let $ \{ d\tau^L \}, \{ d \tau^J \}$ be holomorphic frames on $S^\ell T_M^*$ and $S^{m-\ell} N^*$, respectively.
As $\psi_m^\ell$ is of the form
$$
\psi_m^{\ell} = \sum_{\substack{|L|=\ell \\ |J|=m-\ell}} \psi^\ell_{LJ} (\tau) d \tau^L \otimes d \tau^J,
$$
we have

\begin{equation}\nonumber
    \begin{aligned}
        \Theta(\psi_m^{\ell})(\zeta,w) 
        &= \int_{ \imath(\widetilde{M}) }\left\langle
        \left(\sum_{\ell,j} (1-|\tau|^2)(\delta_{j\ell} - \tau_j\overline \tau_\ell)\frac{w_\ell}{1-\left\langle w, \tau \right\rangle}\frac{\partial}{\partial \tau_j}\right)^m ,
        \sum \psi_{LJ}^{\ell} (\tau)d\tau^L \otimes d \tau^J 
        \right\rangle ~dV_{\widetilde M}\\
        &=\int_{\imath(\widetilde M) }
        \frac{(1-|\tau|^2)^{m-(n+1)}}{(1-\left\langle w, \tau \right\rangle)^m}
        \left\langle
        \left(\sum_j(w_j- \tau_j\left\langle w, \tau\right\rangle)\frac{\partial}{\partial \tau_j}\right)^m,  \sum \psi_{LJ}^\ell (\tau)d\tau^L \otimes d \tau^J
        \right\rangle
        ~dV_{} \\
         &=\int_{ \imath(\widetilde M) }
        \left ( \frac{(1-|\tau|^2)^{m-(n+1)}}{(1-\left\langle w, \tau \right\rangle)^m}
        \sum_{\substack{|L|=\ell\\ |J|=m-\ell} }\psi_{LJ}^{\ell} (\tau)
        \prod_{\mu=1}^N (w_\mu-\tau_\mu\left\langle w, \tau \right\rangle)^{i_\mu} \right) \Bigg |_{\imath(\widetilde M)} dV_{} \\
        \end{aligned}
        \end{equation}
        where $dV$ is the Euclidean volume form of $\mathbb C^n$ and
        \begin{equation}\nonumber
        i_\mu:=
        \begin{cases}
        \ell_\mu \quad \text{when}  \; 1 \leq \mu \leq n, \\
        j_\mu  \quad \text{when} \; n+1 \leq \mu \leq N.
        \end{cases}
        \end{equation}
        Here, we use \eqref{normalize 1}. 
Then, by expanding each term on $\imath(\widetilde M)$ as 
$$
\frac{1}{(1-\left\langle w, \tau \right\rangle)^m}
=\sum_{s=0}^\infty \frac{(m+s-1)!}{(m-1)!\, s!}\left\langle w,\tau\right\rangle^s,
$$
$$
\psi_{LJ}^{\ell}(\tau) 
= \sum_{k=0}^\infty\sum_{|K|=k}
\frac{1}{K!}\frac{\partial^K\psi_{LJ}^\ell}{\partial\tau^K }(0)\tau^K,
$$ 
$$
\prod_{\mu=1}^n (w_\mu-\tau_\mu\left\langle w, \tau \right\rangle)^{i_\mu}
=\sum_{r=0}^{\ell} \sum_{\substack{A+B=L\\ |A|=r}}\frac{(-1)^{\ell-r} L!}{A!\,B!}w^A\tau^B \left\langle w,\tau \right\rangle^{\ell-r},
$$
and 
$$
\prod_{\mu=n+1}^{N} (w_\mu - \tau_\mu \langle w, \tau \rangle )^{i_\mu} = w^J
$$
on $\imath(\widetilde M)$,  by \cite[Lemma~1.11]{Z} we obtain 
        \begin{equation}\nonumber
            \begin{aligned}
            \Theta(\psi_m^{\ell})(\zeta,w)
        &= \sum_{\substack{|L|=\ell\\ |J|=m-\ell} }\sum_{k=0}^\infty\sum_{|K|=k} \sum_{r=0}^{\ell}  \sum_{\substack{A+B=L\\ |A|=r}}\frac{(-1)^{{\ell}-r} L! }{A!\,B!}\frac{(m+k-1)!}{(m-1)!\, k!K!}\frac{\partial^K\psi_{LJ}^{\ell}}{\partial\tau^K }(0)w^A w^J\\
        &\quad\quad\quad\quad\quad\quad\quad\quad
        \quad\quad\quad\quad\quad
        \cdot\int_{\mathbb{B}^n}
        (1-|\tau|^2)^{m-(n+1)}
        \tau^{K+B} \left\langle w,\tau \right\rangle^{\ell-r+k} dV_{\mathbb{B}^n}.
        \end{aligned}
        \end{equation}
Here, we used that if $P\neq Q$, then $\int_{\mathbb B^n} (1-|\tau|^2)^m \tau^P\overline\tau^QdV=0$. 
As a consequence,  by using the equality 
$$
\int_{\mathbb B^n} (1-|\tau|^2)^m\, \tau^P\,\overline\tau^{\,P}\, dV(\tau)
= \pi^n
\frac{\Gamma(m+1)\,\prod_{j=1}^n \Gamma(p_j+1)}
{\Gamma(m+|P|+n+1)},
$$
we have
        \begin{equation}\nonumber
            \begin{aligned}
            \Theta(\psi_m^{\ell})(\zeta,w)
         &=\sum_{\substack{|L|=\ell\\ |J|=m-\ell} }\sum_{k=0}^\infty\sum_{|K|=k} \sum_{r=0}^{\ell} \sum_{\substack{A+B=L\\ |A|=r}}\frac{(-1)^{\ell-r} L! }{A!\,B!}\frac{(m+k-1)!}{(m-1)!\, k!K!}\frac{\partial^K\psi^\ell_{LJ}}{\partial\tau^K }(0)w^{L+K} w^J \\
         & 
        \quad\quad\quad\quad\quad\quad\quad\quad
        \quad\quad\quad\quad\quad
        \cdot \frac{(\ell-r+k)!}{(B+K)!}\int_{\mathbb B^n}
        (1-|\tau|^2)^{m-n-1}
        \tau^{K+B}  \overline\tau^{K+B}~dV\\
        &=\frac{\pi^n(m-n-1)!}{(m-1)!\, }\sum_{\substack{|L|=\ell\\ |J|=m-\ell}} \sum_{k=0}^\infty\sum_{|K|=k} {\sum_{r=0}^{\ell} \sum_{\substack{A+B=L\\ |A|=r}}\frac{(-1)^{\ell-r} L! }{A!\,B!}
        \frac{(\ell-r+k)!}{(m+\ell-r+k-1)!} }\\
        &\quad\quad\quad\quad\quad\quad\quad\quad
        \quad\quad\quad\quad\quad\quad\quad\quad
        \cdot \frac{(m+k-1)!}{k!K!}\frac{\partial^K\psi_{LJ}^{\ell}}{\partial\tau^K }(0)w^{L+K}w^J \\
        &=\frac{\pi^n (m-n-1)!}{(m-1)! (m-2)!} \sum_{\substack{|L|=\ell\\ |J|=m-\ell} }\sum_{k=0}^{\infty} \sum_{|K|=k} \frac{ {(m+\ell-2)!} (m+k-1)!}{(m+\ell+k-1)! K!} \frac{\partial^K \psi_{LJ}^{\ell}}{\partial \tau^K} (0) w^{L+K}w^J.
        \end{aligned}
        \end{equation}
        
On the other hand, 
by \cite{LS2}  for each $\psi^\ell_{m}$, there exist the  minimal solutions $\varphi_{m+s}^\ell$ with $s\in \mathbb N$ of the following recursive $\bar \partial$-equation:
$$
\bar \partial_M \varphi_{m+s}^\ell = -(m+s-1) \mathcal{R}_{G} \varphi_{m+s-1}^\ell
$$
with 
\begin{equation}\label{initial setting}
\varphi_{m}^\ell = \psi_m^\ell
\end{equation}
for the corresponding raising operator $\mathcal R_G$.
Write 
\begin{equation}\label{form 1}
\varphi_{m+s-1}^{\ell} = \sum_{\substack{|L|=\ell+(s-1) \\ |J|=m-\ell} } f_{LJ}^\ell e^L \otimes e^J.
\end{equation}
Then
\begin{equation}\nonumber
\begin{aligned}
    \varphi_{m+s}^{\ell}
    &=-\frac{m+s -1}{E^{\ell}_{m,s-1} + (\ell+s-1) + (m+s-1)}
    \bar \partial^{*}  \left( \mathcal R_G(\varphi_{m+s-1}^\ell  ) \right)    .
\end{aligned}
\end{equation}
By repeating the process, using Corollary ~\ref{adjoint on ball}, we obtain

\begin{equation}\nonumber
\begin{aligned}
    \varphi_{m+s}^{\ell}(0) 
    &= \prod_{\alpha=1}^{s} \frac{m+\alpha-1}{E_{m,\alpha-1}^{\ell} + (\ell+\alpha-1)+(m+\alpha-1)} S^{\alpha} (\varphi^\ell_m) \\
    &= (-1)^s \prod_{\alpha=1}^s \frac{m+\alpha-1}{E_{m,\alpha-1}^{\ell} + (\ell+\alpha-1)+(m+\alpha-1)} 
    \sum_{LJ}\sum_{|K|=s} { \frac{s!}{K!} } \frac { \partial^s \varphi_{LJ}^\ell } {\partial \tau^K}(0) e^{L+K} e^J\\
    &= (-1)^s  \prod_{\alpha=1}^s \frac{(m+\alpha-1)}{E^\ell_{m,\alpha}}
    \sum_{LJ}\sum_{|K|=s} \frac{s!}{K!}\frac{\partial^s \varphi^\ell_{LJ}}{\partial \tau^K}(0) e^{L+K} e^J.
    \\
    &= (-1)^s \prod_{\alpha=1}^s \frac{(m+\alpha-1)}{\alpha(m+\ell+\alpha-1)}  \sum_{LJ}\sum_{|K|=s} \frac{s!}{K!}\frac{\partial^s \varphi^\ell_{LJ}}{\partial \tau^K}(0) e^{L+K}e^J \\
    &= (-1)^s \frac{(m+s-1)!(m+\ell-1)!}{ (m-1)!(m+\ell+s-1)!} \sum_{LJ} \sum_{|K|=s} \frac{1}{K!}\frac{\partial^s \psi^\ell_{LJ}}{\partial \tau^K}(0) e^{L+K}e^J
\end{aligned}    
\end{equation}
by \cite[pp. 17]{LS2}, \eqref{initial setting} and \eqref{form 1}.

Finally, by $T_0 w = - w$ we have 
\begin{equation}\nonumber
\begin{aligned}
\Phi (\psi_m^\ell)(\zeta,w) &= \sum_{s=0}^{\infty} \sum_{\substack{|L|=\ell+s\\ |J|=m-\ell}}^\infty f_{LJ}^\ell(0) (-w)^{L} (-w)^{J} \\
    &=(-1)^m \frac{{(m+\ell-1)!}}{(m-1)!}
    \sum_{s=0}^\infty    \frac{(m+s-1)!}{(m+\ell+s-1)!}\sum_{LJ}
    \sum_{|K|=s} 
    \frac{1}{K!}\frac{\partial^s \psi_{LJ}^\ell}{\partial \tau^K}(0) w^{L+K}w^J. \\
    &= \frac{c_{n,m}(m+\ell-1)}{2m-1} \Theta (\psi_m^\ell) (\zeta,w)
\end{aligned}
\end{equation}
where we use $\imath(\zeta)=0$. Therefore, the desired theorem is proved.
\end{proof}

\end{document}